\documentclass[12pt,reqno]{amsart}
\usepackage[margin=1in]{geometry}

\usepackage{hyperref}
\usepackage{amsthm}
\usepackage{amssymb}
\usepackage{amsxtra}
\usepackage{kopplagarias-38} 

\usepackage{epsfig}
\usepackage{verbatim}
\usepackage{cleveref}
\usepackage{tikz-cd}
\usepackage{autonum} 
\usepackage{cite} 
\usepackage{hhline}

\numberwithin{equation}{section}

\theoremstyle{plain}
\newtheorem{theorem}{Theorem}[section]
\newtheorem{thm}[theorem]{Theorem}
\newtheorem{prop}[theorem]{Proposition}
\newtheorem{cor}[theorem]{Corollary}

\theoremstyle{definition}
\newtheorem{defn}[theorem]{Definition}

\newtheorem{prob}[theorem]{Problem}

\theoremstyle{remark}
\newtheorem{rmk}[theorem]{Remark}

\AddToHook{env/prop/begin}{\crefalias{theorem}{prop}}
\AddToHook{env/cor/begin}{\crefalias{theorem}{cor}}
\AddToHook{env/lem/begin}{\crefalias{theorem}{lem}}
\AddToHook{env/que/begin}{\crefalias{theorem}{que}}
\AddToHook{env/conj/begin}{\crefalias{theorem}{conj}}
\AddToHook{env/assu/begin}{\crefalias{theorem}{assu}}
\AddToHook{env/defn/begin}{\crefalias{theorem}{defn}}
\AddToHook{env/nota/begin}{\crefalias{theorem}{nota}}
\AddToHook{env/cond/begin}{\crefalias{theorem}{cond}}
\AddToHook{env/egz/begin}{\crefalias{theorem}{egz}}
\AddToHook{env/prob/begin}{\crefalias{theorem}{prob}}
\AddToHook{env/rmk/begin}{\crefalias{theorem}{rmk}}

\begin{document}

\title[Unit-generated orders I]{Unit-generated orders of real quadratic fields\\ I. Class number bounds}
\author{Gene S. Kopp}
\address{Department of Mathematics, Louisiana State University, Baton Rouge, LA, USA}
\email{kopp@math.lsu.edu}
\author{Jeffrey C. Lagarias}
\address{Department of Mathematics, University of Michigan, Ann Arbor, MI, USA}
\email{lagarias@umich.edu}


\date{April 21, 2026}

\begin{abstract} 
Unit-generated orders of a quadratic field are orders of the form $\mathcal{O} = \mathbb{Z}[\varepsilon]$, where $\varepsilon$ is a unit in the quadratic field. If the order $\mathcal{O}$ is a maximal order of a real quadratic field, then the quadratic number field is necessarily of a restricted form, being of narrow Richaud--Degert type. However, every real quadratic field contains infinitely many distinct unit-generated orders. They are parametrized as $\mathcal{O} = \mathcal{O}_{n}^{\pm}$ having quadratic discriminants $\Delta(\mathcal{O}) = \Delta_{n}^{+} = n^2 - 4$ (for $n \geq 3$) and $\Delta(\mathcal{O}) = \Delta_{n}^{-} = n^2 + 4$ (for  $n \geq 1$). We show the (wide or narrow) class numbers of unit-generated orders satisfy $\log \left|{\rm Cl}(\mathcal{O})\right| \sim \log \frac{1}{2}\left|\Delta(\mathcal{O})\right|$ as $\left|\Delta(\mathcal{O})\right| \to \infty$, using a result of L.-K. Hua. We deduce that there are finitely many unit-generated quadratic orders of class number one and finitely many unit-generated quadratic orders whose class group is $2$-torsion. We classify all unit-generated real quadratic orders having class number one. We provide numerical lists of quadratic unit-generated orders whose class groups are $2$-torsion for $\Delta \leq 10^{10}$, for both wide and narrow class groups. These lists are conjecturally complete for all $\Delta$. 
\end{abstract}

\maketitle

\section{Introduction}\label{sec:intro}

An order $\OO$ of a number field $K$ is a subring of its algebraic integers $\OO_K$ that is a $\Z$-module of rank $n= [K: \Q]$. 
Orders $\OO$ of quadratic fields are uniquely determined by their discriminant $\Delta$, 
which can be any nonsquare integer congruent to $0$ or $1$ modulo $4$; see \cite[Theorem 5.1.7]{Halter-Koch13}. 
We write $\OO_{\Delta} = \Z\!\left[\frac{\Delta + \sqrt{\Delta}}{2}\right]$ for the order of discriminant $\Delta$. 

This paper focuses on a special subclass of orders of quadratic fields.

\begin{defn}
An order $\OO$ of a quadratic number field $\Q(\sqrt{D})$ is a {\em unit-generated order} if it is additively generated as a $\Z$-module by its set of units $\OO^{\times}$. 
\end{defn}

We show that unit-generated orders are precisely those with discriminants falling in two parametric families.
\begin{enumerate}
\item
$\OO_{\Delta}$ 
with $\Delta = \Delta^{+}_n = n^2-4$ has the generating unit $\e_n^{+} = \frac{1}{2} (n + \sqrt{n^2- 4})$, 
which has norm $\Nm(\e_n^{+}) = +1$.
\item
$\OO_{\Delta}$ 
with $\Delta = \Delta^{-}_n = n^2+4$ has the generating unit $\e_n^{-} = \frac{1}{2} (n + \sqrt{n^2+4})$, 
which has norm $\Nm(\e_n^{-}) = -1$.
\end{enumerate}
In the notation $\Delta^{\pm}_n$, the sign $\pm$ indicates the norm of the associated generating unit. 

Every real quadratic field 
contains infinitely many 
orders that are unit-generated orders, 
as (taking $\Delta = \disc(K)$) it contains infinitely many units $\e>1$, with $\e = \frac{1}{2}(a+ b\sqrt{\Delta}) = \frac{1}{2} (a + \sqrt{a^2 \mp 4})$ and $a^2-b^2\Delta = \pm 4$. 

The main object of this paper is to prove results on class numbers in the two
parametric  families  of unit-generated quadratic orders.  These include: 
\begin{enumerate}
\item[(1)]
An asymptotic estimate for the size of class numbers of $\OO_{n^2 \mp 4}$ as $n \to \infty$.
The estimate is ineffective in general, but includes an effective computability result for subclasses of
unit-generated orders in which the associated fundamental discriminants are bounded.
\item[(2)]
A complete classification of unit-generated
quadratic orders having class number one.
\item[(3)]  
An (ineffective) proof of finiteness of the number of unit-generated orders whose wide class group is $2$-torsion.
This set includes all unit-generated orders having one class per genus.
We tabulate a list of unit-generated quadratic orders whose wide class group is $2$-torsion, complete for $\Delta < 10^{10}$,
and identify all such orders having one class per genus for $\Delta < 10^{10}$.
The list is conjecturally complete for both problems.
\end{enumerate}

The family $\Delta^{+}(n)$ 
is of special interest 
because
its members appear phenomenologically in the structure of
SIC-POVMs
in quantum information theory; see \Cref{subsec:13}.
Numerical evidence  for the connection between SIC-POVMs and the orders of discriminant $\Delta^{+}(n)$
is detailed in \cite{kopplagarias2}. 

\subsection{Results}\label{subsec:11}

Section \ref{sec:unit-generated} deals with taxonomy. 
We show that every unit-generated quadratic order is one of:
\begin{enumerate}
\item The real quadratic order $\OO_{\Delta^+_n}$ for $\Delta^+_n = n^2 - 4$ with $n \geq 3$,
\item The real quadratic order $\OO_{\Delta^-_n}$ for $\Delta^-_n = n^2 + 4$ with $n \geq 1$,
\item The imaginary quadratic order $\OO_{\Delta^+_1} = \OO_{-3} = \Z\!\left[\frac{1+\sqrt{-3}}{2}\right]$, or
\item The imaginary quadratic order $\OO_{\Delta^+_0} = \OO_{-4} = \Z[\sqrt{-1}]$.
\end{enumerate}
The two imaginary quadratic orders are special, as they are
the only two imaginary quadratic orders having extra units beyond $\{\pm1\}$.
In general we restrict attention to real quadratic orders.

All of the orders in these parameter ranges are distinct, with one exception:\ The orders $\OO_{\Delta^+_3} = \OO_{\Delta^-_1}$ are both the maximal order of $\Q(\sqrt{5})$.
For the remaining parameter values in the range $n \ge 0$, $\Delta_{2}^{+}= 0$ and $\Delta_0^{-}=4$ are not discriminants of quadratic orders.

\subsubsection{Class number bounds}\label{subsubsec:121}
Let $\Cl(\Delta) = \Cl(\OO_\Delta)$ denote the (wide) class group of the quadratic order of discriminant $\Delta$,
and let
$h_{\Delta}  
= \abs{\Cl(\OO_\Delta)}$ 
denote the class number of the order.
We write $\Delta = n^2 \mp 4 = f_{\Delta}^2\Delta_0$ where
$\Delta_0$ is a fundamental discriminant, the discriminant of the associated maximal order $\OO_K$,
and $f=f_{\Delta}$ is the conductor of the order.

It is known that there exist infinitely many (non-maximal) real quadratic orders of class number one. 
The first examples were  given by Dirichlet \cite{Dirichlet:1856} in 1855 in the context of binary quadratic forms.
For example, 
$\Delta= 5^{2k+1}$ for $k \ge0$ are discriminants of orders in $\Q(\sqrt{5})$, which all have $h_{\Delta}=1$.
(See \cite[Theorem 5.9.12]{Halter-Koch13} and \cite[Section 9, Aufgabe 5]{Zagier81}). 

\begin{thm}\label{thm:main}
For unit-generated orders having $\Delta_n^{\pm}= n^2 \mp 4$,
\begin{equation}\label{eq:main1}
\log \abs{\Cl(\Delta_n^{+})} = \log n + o(\log n) \quad \mbox{as} \quad n \to \infty,
\end{equation}
and
\begin{equation}\label{eq:main2}
\log \abs{\Cl(\Delta_n^{-})} = \log n + o(\log n) \quad \mbox{as} \quad n \to \infty.
\end{equation}
\end{thm} 

This result is proved as  \Cref{thm:bs00}(1).
It is an immediate consequence of a theorem of Hua \cite[Theorem 12.15.4]{Hua:1982}, 
which gives an extension to quadratic orders of Siegel's original version of the Brauer--Siegel theorem.
The error term in this result is ineffective.

\Cref{thm:bs00}(2) shows that if one restricts to families of unit-generated orders having bounded 
associated fundamental discriminants $(\Delta_n^{\pm})_0 \le N$,
then there is an effective remainder term $O_{\!N}(\log\log n)$.
Bounded families arise in the context of SIC-POVMs (discussed in \Cref{subsec:13}), 
for example, in the Fibonacci--Lucas family of SIC-POVMs studied in \cite{GrasslS:2017}.

\subsubsection{Class number one}\label{subsubsec:122}

We give a complete classification of all unit-generated quadratic orders having class number one.

Classifying the set of discriminants of unit-generated real quadratic orders having class number one
is analogous to the famous Gauss problem of classifying all discriminants $\Delta<0$ having class number 
one. 
The imaginary quadratic discriminant problem for maximal orders was solved by the Heegner--Stark--Baker Theorem. 
Stark \cite{Stark:07} notes that the non-maximal order case for imaginary quadratic fields is solved as a consequence of the maximal order case,
giving the four discriminants $\{-12, -16, -27, -28\}$, using the fact that non-maximal orders can have class number one
only if the maximal order also has class number one.

Our first result gives a complete classification for unit-generated quadratic orders that are maximal orders. 

\begin{thm}\label{thm:M-classno-1}
The maximal quadratic orders of discriminant $\Delta_n^{+}= n^2 - 4$ 
having class number one are those for $n \in \{0, 1, 3, 4, 5, 9, 21\}$, having 
$\Delta \in \{-4, -3, 5,  12,  21, 77,  437\}$.
The maximal quadratic orders of discriminant
$\Delta_n^{-}= n^2 + 4$ having class number one are 
those for $n \in \{1, 2, 3, 5, 7, 13, 17\}$, having 
$\Delta \in \{ 5, 8, 13,  29,  53, 173, 293\}$.
\end{thm}

\Cref{thm:M-classno-1} is proved in \Cref{subsec:41a}
by combining four results. The first is 
a major result of Bir\'{o} \cite{Biro03a} handling the case $n^2+4$ with $n$ odd. The second is a  
a result of Byeon, Kim, and Lee \cite{BKL} handling the case $n^2-4$ with $n$ odd,
using a modification of Bir\'{o}'s  method.
This paper completes the classification in 
the remaining two cases, $n^2+4$ with $n$ even  
(\Cref{thm:RQO1})
and $n^2-4$ with $n$ even
(\Cref{thm:RQO2}),
using genus theory.

The next result gives a complete classification for unit-generated orders that are non-maximal orders.

\begin{thm}\label{thm:M-classno-2}
The non-maximal quadratic orders of discriminant $\Delta_n^{+}= n^2 - 4$ 
having class number one are those for $n \in \{6, 7, 11\}$, having 
$\Delta \in \{32, 45, 117\}$.
The non-maximal quadratic orders of discriminant
$\Delta_n^{-}= n^2 + 4$ having class number one are 
those for $n \in \{4, 8, 11\}$, having 
$\Delta \in \{ 20, 68, 125\}$.
\end{thm}

The proof of \Cref{thm:M-classno-2} builds on \Cref{thm:M-classno-1}. The idea of the proof is similar to the imaginary quadratic order case, in which all the over-orders of a non-maximal imaginary quadratic order having class number one must also be unit-generated orders. In the real quadratic unit-generated case, however, there is an additional exceptional case; see Section \ref{subsec:42a}.
The exceptional case is covered by a major result of Bir\'{o} \cite{Biro03b} solving Chowla's conjecture for maximal orders with $\Delta=4n^2+1$.

\subsubsection{One class per genus and $2$-torsion class group}\label{subsubsec:123}

The genus theory of Gauss extends to all (real and imaginary) quadratic orders. 
The principal genus theorem of Gauss asserts that the (narrow) ideal classes in the principal genus
are the squares of all ideal classes.
Two ideal classes $[\aa]$ and $[\bb]$ in the narrow ideal class group $\Cl^+(\OO)$ 
(defined in \Cref{subsec:23})
correspond under Gauss composition 
to binary quadratic form classes of the same genus if and only if $[\aa\bb^{-1}] \in \Cl^+(\OO)^2$.
The genus group is $\Cl^+(\OO)/\Cl^+(\OO)^2$, which is isomorphic to the $2$-torsion subgroup $\Cl^+(\OO)[2]$.
(See \cite[Sections 5.6 and 6.5]{Halter-Koch13}.)

In \Cref{subsec:52}, we establish the following result regarding the wide class group and its $2$-torsion subgroup.

\begin{thm}\label{thm:twotorsion}
There are finitely many unit-generated real quadratic orders $\OO$ for which $\Cl(\OO_\Delta) = \Cl(\OO_\Delta)[2]$.
\end{thm}

This result is ineffective; it depends on the ineffective result of Hua \cite[Theorem 12.15.4]{Hua:1982}.
It has the following corollary for the one class per genus problem.

\begin{cor}\label{cor:oneclassgenus}
There are finitely many unit-generated quadratic orders $\OO = \OO_\Delta$ for which $\Cl^+(\OO) = \Cl^+(\OO)[2]$.
Equivalently, there are finitely many $\Delta= n^2 \pm 4$ such that the primitive integer binary quadratic forms of discriminant $\Delta$ 
have one class per genus.
\end{cor}
\begin{proof}
There is a surjective map $\Cl^+(\OO) \to \Cl(\OO)$. Thus, if $\Cl^+(\OO)$ is $2$-torsion, then $\Cl(\OO)$ is $2$-torsion. 
So the finiteness of unit-generated $\OO$ for which $\Cl(\OO)$ is $2$-torsion implies the finiteness of such $\OO$ satisfied 
the stronger condition that $\Cl^+(\OO)$ is $2$-torsion.
\end{proof}

\Cref{thm:twotorsion} leads to the following problem.
 
\begin{prob}\label{prob:2}
Determine all unit-generated real quadratic orders $\OO_{\Delta}$ having one class per genus,
that is,
having $\Cl^{+}(\OO_\Delta) = \Cl^{+}(\OO_\Delta)[2]$.
More generally, determine all such orders having 
$\Cl(\OO_\Delta) = \Cl(\OO_\Delta)[2]$.
\end{prob}

Classifying the set of discriminants of unit-generated 
real quadratic orders having one class per genus
is analogous to the Gauss problem of classifying all discriminants $\Delta <0$ having one class
per genus, raised in Article 304 of \textit{Disquisitiones Arithmeticae}; see \Cref{sec:5} for details.
We present in Section \ref{sec:5} tables of discriminants of all such quadratic orders having $\Cl(\OO_\Delta) = \Cl(\OO_\Delta)[2]$
for $\Delta < 10^{10}$, which conjecturally give the complete list of all such quadratic orders.

\subsection{Applications}\label{subsec:13}
SICs or SIC-POVMs (symmetric, informationally complete, positive operator-valued measures) are generalized quantum measurements corresponding to arrangements of $d^2$ equiangular complex lines in $\C^d$. There is an empirically observed surprising
connection between SICs in dimension $d$ and the quadratic field $\Q(\sqrt{(d+1)(d-3)})$. 

The connection was observed in work of Appleby, Yadsen-Appleby, and Zauner \cite{AYZ:13} in 2013; see also \cite{AFMY:17,AFMY:20}. 
In particular, SICs were connected directly with the unit-generated quadratic order $\OO_{\Delta}$
with discriminant  $\Delta^{+}_{d-1} = (d-1)^2-4 = (d+1)(d-3)$
in \cite{kopplagarias2}, where it was noted that the number of known SICs in dimension $d$, for $d \leq 90$, equals the sum of the class numbers of the overorders of $\OO_{\Delta_{d-1}^+}$.
Moreover, the fields in which the (appropriately scaled) vectors defining a SIC live 
appear to be certain ray class fields of those orders (as defined in \cite{kopplagarias}). 
The conjectural connection to class field theory
can be made very precise and explicit and is related to the Stark conjectures \cite{AFK:25,koppsic,ABGHM:22,BGM:24}.

The bounds on class numbers in this paper, combined with conjectures on SICs formulated in
\cite{kopplagarias2}, imply bounds on the number of equivalence classes of Weyl--Heisenberg covariant SICs in dimension $d$.

\subsection{Prior work}\label{subsec:12a} 

For background on the theory of real quadratic orders, see the books of 
Halter-Koch \cite[Chapter 5]{Halter-Koch13} and Mollin \cite[Chapters 3 and 5]{Mollin:96}.

Class number problems for 
maximal orders of real quadratic fields have been extensively studied.
The terminology that (the maximal order of) a field $\Q(\sqrt{D})$ is of {\em Richaud--Degert type} was coined by Hasse \cite[p.~52]{Hasse:65} in 1965 and states
\begin{equation}
D= m^2 + r, \quad -m < r \le m, \quad r \divides 4m,
\end{equation}
restricting to the cases that $D$ is squarefree. (Note that $D$ need not be a discriminant here; the associated discriminant will be either $\Delta=D$ or $\Delta=4D$.)
The Richaud--Degert condition can also be phrased as a condition stating that the continued fraction expansion of $\sqrt{D}$ has one of several short parametric forms.
(The continued fractions for $n^2 \pm 4$ are particularly simple; see Remark \ref{rem:CF}.)
The fields are named after the 1866 work of Richaud \cite{Richaud:1866} and 1958 work of Degert \cite{Degert:58}
on fundamental units.

In 1996, Mollin \cite[pp.~77--78]{Mollin:96} introduced the terminology
{\em narrow Richaud--Degert type} 
for the special case of $D=m^2+r$ with $r \in \{\pm 1, \pm 4\}$, where $D$ is required to be squarefree. In 1998 Mollin \cite{Mollin:98} considered $\Delta^{\pm}$
as parametric families of quadratic orders, allowing the non-maximal order case.
He gave a sufficient condition for class number one.

Mollin \cite[pp.~77--78]{Mollin:96} 
also introduced the terminology {\em extended Richaud--Degert type}, abbreviated {\em ERD-type}, for
\begin{equation}
D = m^2 + r, \quad r \divides 4m, \,\, r \in \Z,
\end{equation}
removing Hasse's restriction $-m<r<m$; see also \cite{MollinW:90}.
All fields of ERD-type contain unit-generated orders of small conductor (that is, small index in the maximal order). 
This fact can be deduced from \cite[Theorem 3.2.1]{Mollin:96}, which presents the continued fractions of some $\omega$ with $\OO_K = \Z[\omega]$ in all possible cases in which $K$ is of ERD-type. In this paper, we give the largest unit-generated order
in fields of narrow Richaud--Degert type in \Cref{prop:R--D} and \Cref{table:0}.
In these fields, the largest unit-generated order has conductor either $1$ or $2$.

General criteria for real quadratic orders to have class number one were given recently by
Kawamoto and Kishi  \cite[Theorem 2.6]{KawamotoK:24}, \cite[Theorem 3.5]{KawamotoKT:24}.
Caeiro and Darmon \cite[Main Theorem, Theorem 25, and Theorem 28]{CD:25} present conditional results 
classifying all unit-generated orders of class number one 
modulo their Conjecture 7. Their Conjecture 7 asserts that the image of a real multiplication (RM) point 
under a certain rigid analytic cocycle may be used to define a global point on an elliptic curve defined over a ring class field, and it also asserts the truth of a Gross--Zagier formula in this context. 

\subsection{Contents of paper}\label{subsec:14}

\Cref{sec:unit-generated} discusses 
parametrizations of unit-generated orders and definitions 
the of class groups and class numbers.

\Cref{sec:3} proves asymptotic results on class numbers $h_{\Delta} = \abs{\Cl(\OO_{\Delta})}$ for $\Delta = n \pm 4$
as $n \to \infty$.  It proves effective bounds for class numbers restricted to families 
of such orders that
have bounded associated fundamental discriminants.

\Cref{sec:4} classifies all discriminants  $\Delta= n^2 \pm 4$ of orders having $h_{\Delta}=1$.

\Cref{sec:5} 
proves (ineffectively) the finiteness of 
the set of $\Delta= n^2 \pm 4$ such that $\Cl(\OO_\Delta)$ is $2$-torsion.
It thus proves finiteness of the set of such $\Delta$ with one class per genus.
It gives lists of known unit-generated order discriminants
for which the (wide) class group is $2$-torsion; these lists are conjecturally complete and are known to be complete for discriminants below $10^{10}$. 
It also determines the sublist of those discriminants below $10^{10}$ having one class per genus (those whose narrow class group is $2$-torsion).

\Cref{sec:6aa} makes concluding remarks.

\section{Basic properties of unit-generated orders} \label{sec:unit-generated}

The discriminant of a quadratic order $\OO = \OO_\Delta$ is written (following Halter-Koch \cite{Halter-Koch13}) as
\begin{equation}
\Delta = f^2 \Delta_0,
\end{equation}
where $\Delta_0$ is the unique fundamental discriminant of the same sign dividing $\Delta$, and $f \ge 1$.
For real quadratic fields, $\Delta_0>0$ (and for imaginary quadratic fields, $\Delta_0<0$).
The number $f$ is called the \textit{conductor} and is equal to the index of $\OO$ in the maximal order: $[\OO_{\Delta_0} : \OO] = f$.

If $D_0 \neq 1$ is any squarefree integer, we associate to it 
the {\em fundamental discriminant}
\begin{equation}
\Delta_0 = \begin{cases} 
D_0 & \mbox{if} \quad D_0 \equiv 1 \Mod{4}, \\
4D_0 & \mbox{if} \quad D_0 \equiv 2, 3 \Mod{4}.
\end{cases}
\end{equation}
Following Mollin \cite[p.~4]{Mollin:96}, we  
term 
$D_0$ the {\em fundamental radicand} of $\Delta_0$ (or of $\Q(\sqrt{\Delta_0})$).

\subsection{Parameterizations of unit-generated orders}\label{subsec:21}

We first deal with imaginary quadratic unit-generated orders and then with real quadratic unit-generated orders. 

\begin{prop}\label{prop:21}
There are exactly two imaginary quadratic unit-generated orders, which
are the maximal orders $\OO_K$ of $K =\Q(\sqrt{-1})$, given
by $\Z[\sqrt{-1}]$, and that of $K=\Q(\sqrt{-3})$, given by
$\Z[\frac{1+\sqrt{-3}}{2}]$.
\end{prop}
\begin{proof}
The unit-generated orders are $\OO= \Z[\OO^{\times}]$. If $\OO$ contains only the units $\pm 1$,
then $\Z[\OO^{\times}] = \Z$ is not a quadratic order. 
The only quadratic orders having
more than the units $\pm 1$ must be orders containing extra roots of unity, which
are therefore orders in either $K = \Q(\sqrt{-1})$ or $K= \Q(\sqrt{-3})$.
Adjoining an extra root of unity in these cases always gives the maximal order. 
\end{proof}

There are two natural ways to list all the unit-generated real quadratic orders. The first way uses $(\Delta_0, j)$
with $\Delta_0$ being the fundamental discriminant
of the associated quadratic field together with a power $j \ge 1$  of the fundamental unit $\varepsilon_{\Delta_0}$ of $\OO_K$,
which is the smallest unit $\varepsilon >1$ in $\OO_K$. It parameterizes all the unit generated orders in a fixed quadratic fields $K$,
one at a time.

\begin{prop}\label{prop:22}
Let $K$ be a real quadratic field of fundamental discriminant $\Delta_0$.
\begin{itemize}
\item[(1)]
All unit-generated orders $\OO$ of $K$ are generated as $\OO= \Z[\e] = \Z + \e\Z$, where $\e$ is the smallest unit $\e >1$ contained in $\OO$. 
\item[(2)] 
If $\e_{\!\Delta_0} > 1$ is the fundamental unit of $\OO_K$, then for $j \ge 1$, the orders
\begin{equation}
\OO(\Delta_0, j) := \Z[\e_{\!\Delta_0}^j] = \Z + \e_{\!\Delta_0}^j\Z
\end{equation}
comprise  the full set of unit-generated orders in $K$. 
The orders $\OO(\Delta_0, j)$ are all distinct, with the single exception that $\OO(5,1)= \OO(5,2)$, where $\OO(5, 1) = \Z[\frac{1+\sqrt{5}}{2}]$ and $\OO(5, 2) = \Z[\frac{3+\sqrt{5}}{2}]$.
\item[(3)]
If $j \divides k$, then $\OO(\Delta_0, k) \subseteq \OO(\Delta_0, j)$.
Conversely, if $\OO(\Delta_0, k) \subseteq \OO(\Delta_0, j)$, then $j \divides k$, with the single exception that $\OO(5, 1) \subseteq \OO(5,2).$
\end{itemize}
\end{prop}
\begin{proof}
(1) Suppose $\OO = \Z[\OO^{\times}]$. Let $\e_{\!\Delta_0}$ be the fundamental unit of the maximal order $\OO_K$ of the field $K$ generated by $\OO$.
Let $\e := \e_{\!\Delta_0}^j > 1$ be the smallest positive
unit belonging to $\OO$. Then we claim $\OO = \Z+\e\Z$. We certainly have $\Z+\e\Z \subseteq \OO$.
The unit satisfies the minimal polynomial $\e^2 - t\e \pm 1 = 0$ for $t = \Tr_{K/\Q}(\e) \in \Z$, so $\e^{-1} = \pm t \mp \e \in \Z+\e\Z$. Also, $\e^{k+1} = t\e^k \mp \e^{k-1}$, so a straightforward induction argument shows that $\e^k \in \Z+\e\Z$.
But $\OO^{\times} = \{\pm \e^k: k \in \Z\}$, for any
additional unit $\eta$ in $\OO^{\times}$ would by multiplication by a suitable unit $\pm \e^k$ yield a unit $\eta' = \e_{\!\Delta}^i>1$ having $1\le i < j$, contradicting the minimality of $\e$.
We conclude $\OO \subseteq \Z+\e\Z$, so $\OO=\Z[\e]$. 

(2) All units $\e>1$ of $K$ are given as $\e_{\!\Delta_0}^j$ for some $j \ge 1$; see \cite[Theorem 5.2.1]{Halter-Koch13}.
It follows that $\OO(\Delta_0, j)$ is an exhaustive list, which might however contain duplicate orders. 

To determine the duplicates, we use the fact that for $\Delta_0>5$, the powers of the fundamental unit  $\e_{\!\Delta_0}^j= \frac{t_j + f_j \sqrt{\Delta_0}}{2}$ with $t_j \equiv f_j \Mod{2}$ have sequences $(t_j)_{j \ge 0}$ and $(f_j)_{j \ge 0}$ strictly monotonic increasing \cite[Theorem 5.2.5(2)]{Halter-Koch13}. 
Thus, by (1), $\OO(\Delta_0,j) = \OO_{f_j^2\Delta_0}$, and these are all distinct.
For the exceptional case $\Delta=5$, $\e_{5}= \frac{1+ \sqrt{5}}{2}$ has $\e_5^j= \frac{L_j+ F_j \sqrt{5}}{2}$ given by Fibonacci numbers $F_j$ and Lucas numbers $L_j$, and we have $f_j= F_j$, which are all distinct except for $F_1=F_2=1$. 

(3) Both assertions follow from (2) and the formula $\OO(\Delta_0,j) = \OO_{f_j^2\Delta_0}$.
\end{proof} 

The second parametrization of unit-generated orders
gives a direct parameterization of discriminants $\Delta = f^2 \Delta_0$ of the orders.
It gives the orders in two lists,  
ordered by increasing size of the generating units in the lists,
keeping the norm of the unit fixed, while mixing all the different fields $K$ in the ordering. 
The set of all real quadratic unit-generated orders are enumerated by having 
The discriminants consist of the  
two families $\Delta = \Delta^+_n = n^2-4$ for $n \ge 3$ and $\Delta = \Delta^-_n = n^2+4$, for $n \ge 1$ 
as described in \Cref{sec:intro}. 

\begin{prop}\label{prop:23}
Let $K$ be a real quadratic field of 
fundamental discriminant $\Delta_0$, and let $\OO= \Z[\e]$ be a unit-generated order of $K$, with $\e>1$ its minimal positive unit.
Write $\e= \frac{n+f\sqrt{\Delta_0}}{2}$, with $n, f \ge 1$ and $n \equiv f \Mod{2}$. Then $\OO$ has discriminant
\begin{equation}\label{eqn:disc}
\Delta(\OO) = f^2\Delta_0= n^2 \mp 4,
\end{equation}
where $\Nm_{K/\Q} (\e) = \pm 1$.
The conductor $[\OO_K:\OO]=f$.

Conversely, if $\Delta = n^2 - 4$ for $n \geq 3$ or $\Delta = n^2 + 4$ for $n \geq 1$, then $\OO_\Delta$ is a real quadratic unit-generated order.
\end{prop}
\begin{proof}
Write $\Gal(K/\Q)=\{1,\sigma\}$.
Note that $n \ge 1$, because $\e+  \sigma(\e) =n >0$, since the conjugate $\sigma(\e) = \pm \e^{-1} > -1.$
By \Cref{prop:22}, $\OO = \Z + \e\Z$. In other words, $\OO = \Z + \frac{n+f\sqrt{\Delta_0}}{2}\Z = \Z + \frac{n+\sqrt{f^2\Delta_0}}{2}\Z = \OO_{f^2\Delta_0}$, so $\Delta(\OO) = f^2\Delta_0$. Additionally,
\begin{equation}
\pm 1 = \Nm_{K/\Q}(\e) = {\left(\frac{n}{2}\right)\!}^2 - {\left(\frac{f}{2}\right)\!}^2\Delta_0 = \frac{n^2-f^2\Delta_0}{4}, 
\end{equation}
so $f^2\Delta = n^2 \mp 4$. The maximal order $\OO_K= \Z+\frac{\Delta_0+\sqrt{\Delta_0}}{2}\Z$, and by inspection the index $[\OO_K:\OO]=f$.

Conversely, set $\Delta = n^2 \mp 4$. Take $\e = \frac{n + \sqrt{n^2 \mp 4}}{2}$. Then $\e$ is a unit, and $\OO_K = \Z[\e]$.
\end{proof} 

\begin{rmk}\label{rem:CF}
There are parameterizations  of the  continued fraction expansions of the 
generating units $\e_n^{\pm}$  of 
the two families of real quadratic orders $\Delta_n^{\pm}$
using the parameter $n$. 
The ordinary continued fractions are
\begin{align}
\e_{n}^{+} &= \frac{1}{2}(n+\sqrt{n^2-4}) 
            = [n-1, \overline{1, n-2}]_+
            = n-1 + \dfrac{1}{1+\dfrac{1}{n-2 + \dfrac{1}{1+\dfrac{1}{n-2+\dfrac{1}{\ddots}}}}},
            \quad n \ge 3. \\
\e_{n}^{-} &= \frac{1}{2}(n+\sqrt{n^2+4}) 
            = [\overline{n}]_+
            = n + \dfrac{1}{n+\dfrac{1}{n+\dfrac{1}{\ddots}}}, \quad n \ge 1.
\end{align}
The minus (Hirzeburch--Jung) continued fractions are
\begin{align}
\e_{n}^{+} &= \frac{1}{2}(n+\sqrt{n^2-4}) 
            = [\ol{n}]_-
            = n - \dfrac{1}{n-\dfrac{1}{n-\dfrac{1}{\ddots}}}, \quad n \ge 3.\\
\e_{n}^{-} &= \frac{1}{2}(n+\sqrt{n^2+4}) 
            = [n+1,\ol{\{2\}^{n-1},n+2}]_{-}, \quad  n \ge 1.
\end{align}
\end{rmk}

\subsection{Relation between narrow Richaud--Degert fields and unit-generated orders}

We establish a result giving the unit-generated order of smallest conductor $f$ in a real quadratic field $K$ of narrow Richaud--Degert type.

\bgroup
\def\arraystretch{1.1}
\begin{table}[h!]
\begin{tabular}{| c | c | c |}
\hline
fundamental radicand & discriminant $\Delta$ of unit-generated order & relation \\
\hhline{|=|=|=|}
$D_0 = m^2 - 4$, $m$ odd & $\Delta = \Delta_0 = n^2 - 4$, $n$ odd & $n = m$ \\
\hline
$D_0 = m^2 - 1$, $m$ even & $\Delta = \Delta_0 = n^2 - 4$, $n \equiv 0 \Mod{4}$ & $n = 2m$ \\
\hline
$D_0 = m^2 + 1$, $m$ even & $\Delta = 4\Delta_0 = n^2 + 4$, $n \equiv 0 \Mod{4}$ & $n = 2m$\\
$D_0 = m^2 + 1$, $m$ odd & $\Delta = \Delta_0 = n^2 + 4$, $n \equiv 2 \Mod{4}$ & $n = 2m$ \\
\hline
$D_0 = m^2 + 4$, $m$ odd & $\Delta = \Delta_0 = n^2 + 4$, $n$ odd & $n = m$ \\
\hline
\end{tabular}
\medskip
\caption{Relationship between narrow Richaud--Degert real quadratic fields and unit-generated orders. Every narrow 
Richaud--Degert field occurs in exactly one row of this table. All maximal real quadratic unit-generated orders, 
as well as some of conductor $f=2$, occur in the table.}
\label{table:0}
\end{table}
\egroup

\begin{prop}\label{prop:R--D}
If $K$ is a real quadratic field of narrow Richaud--Degert type having fundamental discriminant $\Delta_0$ and fundamental radicand $D_0$, then either:
\begin{itemize}
\item[(1)] $\OO_{\Delta_0}$ is a unit-generated order (with conductor $f=1$), or
\item[(2)] $\OO_{4\Delta_0}$ is a unit-generated order (with conductor $f=2$).
\end{itemize}
Case (2) occurs if and only if $D_0 = m^2+1$ and $2 \divides m$. 
Conversely, if $K$ is a real quadratic field whose maximal order is unit-generated, 
then $K$ is of narrow Richaud--Degert type.
\end{prop}
\begin{proof}
The proposition
follows by checking the five cases illustrated in \Cref{table:0}. Inspection of the conditions $D_0 \equiv 1,2,3 \Mod{4}$ and $\Delta_0 \equiv 0,1 \Mod{4}$ shows that the rows of this table exhaust both all $D_0$ of narrow Richaud--Degert type and all $\Delta_0$ of maximal unit-generated quadratic orders.
\end{proof}

\subsection{Class groups of orders}\label{subsec:23}

We define class groups of orders in terms of integral and fractional ideals. 
We follow \cite[Chapter 5]{Halter-Koch13}, especially Section 5.4 and Theorem 5.4.2.
A nonzero integral ideal $\aa$ of the order $\OO_{\Delta}$ has a representation as
a $\Z$-module as $\aa = ea\Z + e\frac{b+ \sqrt{\Delta}}{2}$ with $\Delta= b^2-4ac$,
for unique integers $a, e \ge 1$ and integers $b,c$.
It is {\em $\OO_{\Delta}$-primitive} if $e=1$ and it is {\em $\OO_{\Delta}$-invertible} if and only if
$g := \gcd(a, b, c)$ has $g=1$. It is {\em $\OO_{\Delta}$-regular} if and only if $e=g=1$.
An invertible fractional ideal is a $\Z$-module of the form $r \aa$ with $r \in K^{\times}$,
where $\aa$ is an $\OO_{\Delta}$ invertible integral ideal; it suffices to take $r \in \Q^\times$ to give all invertible fractional ideals.

The class group of an order, defined in terms of invertible fractional ideals, generalizes the usual notion of the class group of (the maximal order of) a number field.
\begin{defn}
The (wide) \textit{class group} (also called the \textit{Picard group} or \textit{ring class group}) of an order $\OO$ in a number field $K$ is
\begin{equation}
\Cl(\OO) := \frac{\{\mbox{invertible fractional ideals of } \OO\}}{\{\mbox{principal fractional ideals $\alpha\OO$, where $\alpha \in K^\times$}\}}.
\end{equation}
\end{defn}
The {\em (wide) class number} $h_{\Delta}= \abs{\Cl(\OO_{\Delta})}$.

\begin{defn}\label{defn:27}
In the real quadratic order case, the \emph{narrow} class group $\Cl^{+}(\OO)$ is the quotient group of the set of invertible fractional ideals by the subgroup of all principal ideals having a totally positive generator. That is,
\begin{equation}
\Cl^{+}(\OO) := \frac{\{\mbox{invertible fractional ideals of } \OO\}}{\{\mbox{principal fractional ideals $\alpha\OO$, where $\alpha \in K^\times$ and $\alpha$ is totally positive}\}}.
\end{equation}
(Note that one can replace $\alpha$ by $-\alpha$ without changing the ideal $\alpha\OO$, so the condition ``$\alpha$ is totally positive'' can be replaced by either the condition ``$\alpha$ is totally negative'' or the condition ``$\alpha$ has positive norm.'')
\end{defn}

We have $\frac{\abs{\Cl^{+}(\OO_{\Delta})}}{\abs{\Cl(\OO_{\Delta})}} = 1$ or $2$, 
the former case occurring 
when there is a unit of norm $-1$ in the order $\OO_\Delta$. 
The {\em narrow class number} $h^{+}_{\Delta}= \abs{\Cl^{+}(\OO_{\Delta})}$.
 
Non-maximal orders always contain non-invertible integral ideals and noninvertible fractional ideals.
If one instead considers all nonzero fractional ideals, quotienting by nonzero principal ideals (which are always invertible), 
one obtains a monoid (that is, a semigroup with identity), called  the  (wide) {\em class monoid} $\Clm(\OO)$. 
The class monoid is a group when $\OO = \OO_K$ and is not a group otherwise.
We consider class monoids for unit-generated orders in a sequel \cite{kopplagarias24b}. 

\section{Growth of the invertible class group of unit-generated real quadratic orders}\label{sec:3}

We give an asymptotic formula for the size of the class group of $\OO_{\Delta}$ in the families $\Delta = \Delta^\pm_n = n^2 \mp 4$.
We deduce the finiteness of the set of unit-generated quadratic orders having one class per genus.

\subsection{Growth of class numbers for arbitrary orders of quadratic fields}\label{subsec3:0}

In 1935 Siegel \cite{siegel:1935}
proved a result on growth of class groups for maximal orders of  imaginary quadratic fields
and a corresponding result on growth of class number times regulator for real quadratic fields. 

\begin{thm}[Siegel's theorem for quadratic fields]\label{thm:siegel}
Let $\Delta$ run over discriminants of maximal orders of quadratic number fields.
\begin{itemize}
\item[(1)]
For $\Delta < 0$, 
\begin{equation}
\lim_{\Delta \to - \infty} \frac{\log h_\Delta}{\log \sqrt{\abs{\Delta}}} = 1
\end{equation}
\item[(2)]
For $\Delta > 0$, 
\begin{equation}
\lim_{\Delta \to  \infty} \frac{\log h_\Delta \log \e_{\Delta}}{\log \sqrt{\Delta}} = 1.
\end{equation}
\end{itemize}
\end{thm} 

This result is a special case of the more general Brauer--Siegel theorem, which applies to any sequence of number fields with discriminants with $\abs{\Delta} \to \infty$.

In 1957 L.-K. Hua gave an  extension of Siegel's theorem to quadratic orders, 
which appeared in English in 1982; see \cite[Theorem 12.15.4]{Hua:1982}.
\begin{thm}[Siegel-type theorem for quadratic orders]\label{thm:hua}
Let $\Delta$ run over discriminants of (not necessarily maximal) orders of quadratic number fields.
\begin{itemize}
\item[(1)]
For $\Delta < 0$, 
\begin{equation}
\lim_{\Delta \to - \infty} \frac{\log h_\Delta}{\log \sqrt{\abs{\Delta}}} = 1
\end{equation}
\item[(2)]
For $\Delta > 0$, 
\begin{equation}
\lim_{\Delta \to  \infty} \frac{\log h_\Delta \log \e_{\Delta}}{\log \sqrt{\Delta}} = 1.
\end{equation}
\end{itemize}
\end{thm} 

An alternate proof is given by Kawamoto and Tomita \cite{KawamotoT:12}.

\subsection{Growth of class numbers of unit-generated orders}\label{subsec:31} 

In the following result, the remainder term in (1) is ineffective,
while the remainder term in (2) is effective.

\begin{thm}\label{thm:bs00} \
\begin{itemize}
\item[(1)]
For $\Delta_n^{\pm} = n^2 \mp 4$, the size of the class group $\Cl(\Delta_n^{\pm})$ obeys the asymptotic formula
\begin{equation}\label{eqn:unit-generated-order-bound}
\log \abs{\Cl(\Delta_n^\pm)} = \log n + o(\log n) \quad \mbox{as} \quad n \to \infty.
\end{equation}
\item[(2)]
For families $\{\OO^{(i)}: \, i \ge 1\}$ of unit-generated orders
whose discriminants $\Delta(\OO^{(i)}) = n_i^2 \pm 4$ satisfy $n_i \to \infty$ as $i \to \infty$,
and which have a bound $(\Delta(\OO_i))_0 \le N < \infty$  on their associated fundamental discriminants, 
\begin{equation}\label{eqn:main4}
\log \abs{\Cl(\OO^{(i)})}= \log n_i + O_{\!N}(\log\log (n_i+20))
\end{equation}
valid for all $n_i \ge 1$, with an effectively computable constant depending on $N$.
\end{itemize}
\end{thm}
\begin{proof}
(1) We study $h_{\Delta}= \abs{\Cl(\OO_{\Delta})}$
for a real quadratic order.
For a unit-generated order with $\Delta = \Delta_n^{\pm}= n^2 \mp 4$,
we have $\log \Delta =2 \log n + O(1)$ for $n \ge 4$, and $\log \e_{\Delta} = \log n +O(1)$ for $n \ge 4$. Consequently, 
we have
\begin{equation}
\frac{\log h_\Delta \log \e_{\Delta}}{\log \Delta} 
= \frac{\log h_\Delta}{\log \Delta} + \frac{\log\log \e_{\Delta}}{\log \Delta}
= \frac{\log h_\Delta}{\log \Delta} + O\!\left(\frac{\log \log n}{\log n}\right).
\end{equation}
Applying Hua's Theorem \ref{thm:hua} for real quadratic orders, 
restricted to  discriminants $\Delta = \Delta_n^{\pm}$ (for each sign separately), we obtain $h_\Delta = \abs{\Cl(\Delta^\pm_n)}$ and 
\begin{equation}
\lim_{n \to \infty} \frac{\log \abs{\Cl(\Delta^\pm_n)}}{\log \sqrt{\Delta_n^{\pm}}} = 1,
\end{equation} 
which gives \eqref{eqn:unit-generated-order-bound}.

(2) Let $\{\sO^{(i)}\}$ be an  infinite subfamily of the full  family $\Delta_n^{\pm}$ of all unit-generated orders,
having the property that all discriminants $\Delta(\sO^{(i)})$ have associated fundamental discriminant 
$\Delta_0(\sO^{(i)}) \le N$, with the bound $N$ known in advance.
Write
$\Delta(\sO^{(i)})=(f_{\sO^{(i)}})^2 \Delta_0(\sO^{(i)})$.
A set $\{\Delta(\sO^{(i)})\}$ of discriminants with 
associated fundamental discriminant bounded by $N$  can be infinite because every fundamental discriminant $\Delta_0$ appears infinitely many
times in the list of $(\Delta_n^\pm)_0$. 
For such a subfamily  necessarily  the conductors $f_{\sO^{(i)}}\to \infty$ as $i \to \infty$.

We write unit-generated order $\Delta_n^{\pm} = f_{\Delta}^2\Delta_0$, where  $\Delta_0$ is fundamental.
We set $f= f_{\Delta}$ and use the formula for the class number of an order given in Neukirch \cite[Theorem 12.12]{Neukirch13}:
\begin{equation}\label{eq:cnorder}
h_\Delta = \frac{h_{\!\Delta_0}}{[\OO_{\!\Delta_0}^\times : \OO_{\!\Delta}^\times]} \frac{\abs{\left(\OO_{\!\Delta_0}/f\OO_{\!\Delta_0}\right)^\times}}{\abs{\left(\OO_\Delta/f\OO_{\!\Delta_0}\right)^\times}}.
\end{equation}
Taking a logarithm on both sides, 
\begin{equation}\label{eq:log-cnorder}
\log h_\Delta
= 
\log\left(\frac{h_{\!\Delta_0}}{[\OO_{\!\Delta_0}^\times : \OO_{\!\Delta}^\times]}\right)
+
\log\left(\frac{\abs{\left(\OO_{\!\Delta_0}/f\OO_{\!\Delta_0}\right)^\times}}{\abs{\left(\OO_\Delta/f\OO_{\!\Delta_0}\right)^\times}}\right).
\end{equation}

We use the identity \eqref{eq:log-cnorder}.
We bound $\abs{\left(\OO_{\!\Delta_0}/f\OO_{\!\Delta_0}\right)^\times}$ above by $f^2$ and below by
\begin{align}
\abs{\left(\OO_{\!\Delta_0}/f\OO_{\!\Delta_0}\right)^\times}
&= f^2 \prod_{\mathfrak{p}|f} \left(1-\frac{1}{\Nm(\mathfrak{p})}\right) 
\geq f^2 \prod_{p|f} \left(1-\frac{1}{p}\right)^2
= \varphi(f)^2,
\end{align}
where the first product is over prime ideals of $\OO_{\!\Delta_0}$ and the second product is over rational primes.
Since $\sqrt{\Delta} = f \sqrt{\Delta_0} \con 0 \Mod{f\OO_{\!\Delta_0}}$, we deduce
that 
$\left(\OO_\Delta/f\OO_{\!\Delta_0}\right)^\times \isom \left(\Z/f\Z\right)^\times$, 
whence $\abs{\left(\OO_\Delta/f\OO_{\!\Delta_0}\right)^\times} = \varphi(f)$. We conclude 

\begin{align}
\frac{f^2}{\varphi(f)} \geq 
\frac{\abs{\left(\OO_{\!\Delta_0}/f\OO_{\!\Delta_0}\right)^\times}}{\abs{\left(\OO_\Delta/f\OO_{\!\Delta_0}\right)^\times}}
\geq \varphi(f).
\end{align}

It is known that 
\begin{equation}
\varphi(f) > \frac{f}{\ell(f)} \quad \mbox{for all} \quad f \ge 2.
\end{equation}
where $\ell(n)$ is the  (effectively computable) function  given by
\begin{equation}
\ell(n) := e^{\gamma} \log\log n + \frac{5}{2\log\log n};
\end{equation}
see \cite[Theorem 15]{RosserS:1962} and \cite[Theorem 8.8.7]{BachS:1996}.
Applying this estimate (on both sides) yields
\begin{equation}
f \,\ell(f) > \frac{\abs{\left(\OO_{\!\Delta_0}/f\OO_{\!\Delta_0}\right)^\times}}{\abs{\left(\OO_\Delta/f\OO_{\!\Delta_0}\right)^\times}} > \frac{f}{\ell(f)},
\end{equation}
Taking logarithms  of these quantities yields the effective bound
\begin{equation}\label{eq:ineq2}
\log f - \log \ell(f)
< \log\!\left(\frac{\abs{\left(\OO_{\!\Delta_0}/f\OO_{\!\Delta_0}\right)^\times}}{\abs{\left(\OO_\Delta/f\OO_{\!\Delta_0}\right)^\times}}\right)
< \log f + \log \ell(f).
\end{equation}

We have 
\begin{equation}\label{eq:BS00}
[\OO_{\!\Delta_0}^\times : \OO_{\!\Delta}^\times] 
= \frac{\log \e_{\!\Delta}}{\log \e_{\!\Delta_0}}.
\end{equation}
Therefore we have
\begin{equation}
\log\!\left(\frac{h_{\!\Delta_0}}{[\OO_{\!\Delta_0}^\times : \OO_{\!\Delta}^\times]}\right)
= 
\log\!\left(\frac{h_{\!\Delta_0}\log \e_{\Delta_0}}{\log \e_{\Delta}}\right).
\end{equation}
Since $\Delta_0$ is bounded, there are a finite set of such $\Delta_0$ 
and of associated log fundamental units $\log \e_{\Delta_0}$, 
hence there is a  positive constant $B$ (depending on $N$) such that
\begin{equation}\label{eq:finineq1}
-B - \log\log\e_\Delta \leq \log\!\left(\frac{h_{\!\Delta_0}}{[\OO_{\!\Delta_0}^\times : \OO_{\!\Delta}^\times]}\right) \leq B - \log\log\e_\Delta.
\end{equation}

Using the bounds \eqref{eq:ineq2}
together with \eqref{eq:finineq1} to estimate the terms on the right side
of the identity \eqref{eq:log-cnorder} yields 
\begin{equation}
\abs{\log h_\Delta - \left(\log\sqrt\Delta_0 + \log f\right)} < B + \log\sqrt\Delta_0 + \abs{\log\log\e_\Delta} + \log\ell(f).
\end{equation}
We have inserted an extra term $\log\sqrt\Delta_0$ on both sides of this expression, noting that $ \Delta_0 \le N$.
On the left-hand side, we have the bound
\begin{equation}
\log\sqrt\Delta_0 + \log f = \log \sqrt{\Delta} = \log n +O(1),
\end{equation}
valid for all $n \ge 4$. Thus we obtain
\begin{equation}\label{eq:bound2}
\abs{\log h_\Delta - \log n} < B' + B + \log N + \abs{\log\log\e_\Delta} + \log\ell(f).
\end{equation}
for an effectively computable constant $B'$.
On the right-hand side of \eqref{eq:bound2}, we have 
\begin{equation}
\log\ell(f)\le\log\log\log (f +20)+O (1) \le \log\log \log (n+20) + O(1), 
\end{equation}
valid for all $f \ge 1$.
Also we have 
\begin{equation}\label{eq:unit-bound}
\log\log\e_\Delta \leq \log\log\e_\Delta^+ = \log\log\frac{n+\sqrt{\Delta_n^\pm}}{2} \le \log\log(n+20) +O(1),
\end{equation}
valid for all $n \ge 1$ where $\Delta_n^{\pm}$ is real quadratic. 
Substituting these  bounds into \eqref{eq:bound2} yields
\begin{equation}\label{eq:finineq2}
\abs{\log h_\Delta - \log n} =  O (\log\log (n + 20) + \log\log\log (n+20) )
\end{equation}
with an effectively computable constant in the $O$-symbol, valid
uniformly for all $n \ge 1$. 

We actually have a sequence of orders $\{ \sO^{(i)} \}$ and
$\Delta( \sO^{(i)}) = \Delta_{n_i}^{ \pm}$ for some sign $\pm$.
So we  conclude by \eqref{eq:finineq2} that 
\begin{equation}\label{eq:bound4}
\log  h_{\Delta(\sO^{(i)})}=
\log h_{\Delta_{n_i}^{\pm}} = \log n_i + O_{\!N}(\log\log (n_i+20)).
\end{equation}
valid for all $n_i \ge 1$ with an effectively computable constant in the big-$O$.
\end{proof}

\section{Unit-generated orders with trivial class group}\label{sec:4}

We classify all unit-generated quadratic orders with class number one.

\subsection{Unit-generated maximal orders with class number one}\label{subsec:41a}

We prove \Cref{thm:M-classno-1} by analysis of cases.
 
\subsubsection{$\Delta^-_n = n^2+4$}\label{subsubsec:411} 
We divide the case $\Delta = \Delta^-_n = n^2+4$ ($n \ge 1$) into two subclasses, according to the parity of $n$.
For $n$ odd, in 1986 Yokoi \cite{Yokoi86} conjectured that all  maximal orders with discriminants $\Delta=n^2+4$ 
that have class number one must have $n \le 17$. In 2003 Bir\'{o} \cite{Biro03a} proved Yokoi's conjecture.

\begin{thm}[Bir\'{o}]\label{thm:Biro} 
If $\Delta = n^2+4$ is the discriminant of a maximal order with odd $n=2m+1 \ge 0$, having class number $h_\Delta=1$, then necessarily $n \le 17$.
The solutions are $n \in \{1, 3, 5, 7, 13, 17\}, $ with discriminants $\Delta \in \{5, 13, 29, 53, 173, 293\}.$
\end{thm}

We resolve the second subcase, of orders having discriminant $\Delta=n^2+4$ with $n$ even, using genus theory.

\begin{thm}\label{thm:RQO1} 
Let $\Delta = n^2+4$ for $n \geq 1$.
\begin{itemize}
\item[(1)] 
If $\OO_\Delta$ is a maximal order with $n=2m$ even, and  with class number $h_\Delta=1$, then $m=1$, giving $n =2$ and  discriminant $\Delta \in \{8\}$.
\item[(2)]
If $n \equiv 2 \Mod{4}$,  
then the class number $h_\Delta$ is even,
regardless of whether $\OO_\Delta$ is maximal or non-maximal.
\end{itemize} 
\end{thm}

To prove the result, we use known criteria from genus theory for oddness of class numbers of quadratic orders.

\begin{prop}\label{prop:HK}
Let $\Delta$ be a quadratic discriminant.
\begin{enumerate}
\item[(1)]
If $\Delta>0$, then the narrow class number $h^{+}_\Delta = \abs{\Cl^+(\OO_\Delta)}$ is odd if and only if
\begin{enumerate}
\item
either $\Delta \in \{p^r, 4p^r \}$ for some $p \equiv 1 \Mod{4}$ and odd $r \in \N$,
\item 
or $\Delta=8$.
\end{enumerate}
In all these cases $\Nm(\e_{\Delta}) = -1.$
\item[(2)]
If $\Delta >0$ then the (wide) class number $h_\Delta$ is odd if and only if 
\begin{enumerate}
\item
either $h^{+}_\Delta$ is odd, 
\item
or $\Delta= \{p^r q^s, 4 p^r q^s\} $ where $p\ne q$ are odd primes, $r$ and $s$ are not both even, $p \equiv \, 3 \Mod{4}$,
and $p^r q^s \equiv \, 1 \Mod{4}$,
\item
or $\Delta = 4p^r \equiv 12 \Mod{16}$ for some prime $p$ and odd $r \in \N$,
\item
or $\Delta \in \{8p^r, 16p^r\}$ for some prime $p \equiv 3 \Mod{4}$ and $r \in \N,$ 
\item
or $\Delta = 32$.
\end{enumerate}
\item[(3)]
If $\Delta$ is a fundamental discriminant, then the (wide) class number $h_\Delta$ is odd if and only if
\begin{enumerate}
\item
either $\Delta= (-1)^{(p-1)/2}p$ for some odd prime $p$,
\item
or $\Delta\in \{4p, 8p\}$ for some prime $p \equiv 3 \Mod{4}$,
\item
or $\Delta = pq$ for some primes $p, q$ with $p \ne q$ and $p\equiv q \equiv 3 \Mod{4}$,
\item
or $\Delta \in \{-4, \pm 8\}$.
\end{enumerate}
\end{enumerate}
\end{prop}

\begin{proof} 
This is the real quadratic case of a result of Halter-Koch \cite[Theorem 5.6.13]{Halter-Koch13}. 
\end{proof}

\begin{proof}[Proof of \Cref{thm:RQO1}]
(1) We first note that $\Delta=8$ has $h_\Delta=1$, coming from $n=2$.
In what follows we suppose $n \ge 4$, so $\Delta \ne 8$.

By hypothesis $n=2m$, so $\Delta= 4m^2+4 = 4(m^2+1)$. There are two cases.\\

{\bf Case 1.} {\it $m=2k$ is even, so $n=4k$, with $k \ge 1$.}\\

Then $\Delta = 4(4k^2+1)$.
Now $\Delta$ cannot be a {fundamental} discriminant since $\Delta' = 4k^2+1$ is already a quadratic discriminant.
So Case 1 is ruled out.\\

{\bf Case 2.} {\it $m=2k+1$ is odd, so $n =4k+2$, with $k \ge 1$.}\\

Then $\Delta= 4 (4k^2+4k +2) =8(2k(k+1)+1)$. 
In particular $\Delta \equiv 8 \Mod{16}$. 
We first show that $h^{+}_\Delta$ is even
by ruling out all cases in \Cref{prop:HK}(1).
Case (a) requires $\Delta \in \{ p^r, 4pr\}$ for $p \equiv 1 \Mod{4}$ with $r$ odd,
and is
ruled out since $8 \divides \Delta$.
Case (b) 
requiring  $\Delta=8$ is excluded  by 
the restriction that $k \ge 1$.

We claim that $h_\Delta$ is also even.
We rule out all cases in the criterion of \Cref{prop:HK}(2)
for it to be odd. Case (a) is ruled out since 
$h^{+}_\Delta$ is even. Case (b) that  $\Delta \in \{p^r q^s, 4 p^r q^s\}$ for $p,q$ odd primes
and case (c) that $\Delta \equiv 12 \Mod{16}$ are  ruled out 
since $8 \divides \Delta$.
The property $16 \nmid \Delta$ rules out one part of case (d), in which $\Delta= 16 p^r$, and case (e) that $\Delta= 32$.
It remains to treat the remaining part of case (d),
that  $\Delta= 8 p^r$ with $p$ prime, $p \equiv 3 \Mod{4}$, and $r \ge 1$. We argue
by contradiction. If equality held, then 
$\Delta=4(4k^2 +4k +2) \equiv 0 \Mod{p}$. However $4k^2+ 4k+2= (2k+1)^2 + 1$ hence $(2k+1)^2 \equiv -1 \Mod{p}$,
which is impossible since $p \equiv 3 \Mod{4}$ by hypothesis.
The claim follows.
We conclude $h_\Delta$ is even in Case 2, whence $h_\Delta \ne 1$. 

(2) The Case 2 argument shows $h_\Delta$ is even for fundamental or non-fundamental $\Delta$.
\end{proof} 

\subsubsection{$\Delta^+_n = n^2-4$}\label{subsubsec:412} 
We divide the case $\Delta= \Delta^{+}=n^2-4$ ($n \ge 0$) into two subcases according to the parity of $n$.
For $n$ odd, in 1987 Mollin \cite{Mollin:87a} conjectured that $h_{n^2-4}>1$ when $\Delta= n^2-4$ is squarefree and $n \ge 21$. 
For $\Delta=n^2-4$ the squarefree assumption is equivalent to $\Delta$ being the discriminant of a maximal order with $n$ odd.
Mollin's conjecture was proved in  2007 by Byeon, Kim, and Lee \cite[Theorem 1.2]{BKL}. 
Their proof uses an adaptation of the general method of Bir\'{o}.

\begin{thm}[Byeon, Kim and Lee]\label{thm:BKL}
If $\Delta=n^2-4$ is the discriminant of a maximal order, with odd $n=2m+1>0$, 
having class number $h_\Delta=1$, then necessarily $n \le 21$.
The solutions are $n \in \{1, 3,  5, 9, 21\}$, with discriminants $\Delta \in \{-3, 5,  21, 77, 437\}$. 
\end{thm}

We resolve the remaining  subcase, of orders having fundamental discriminants $\Delta= n^2-4$ having $n$ even,
using genus theory.

\begin{thm}\label{thm:RQO2} 
Let $\Delta = n^2-4$ for $n \geq 0$.
\begin{itemize}
\item[(1)] If $n=2m$ is even, and 
$\OO_\Delta$ is a maximal order 
with  class number $h_\Delta=1$, then $n \in \{0, 4\}$ 
with discriminant $\Delta \in \{-4, 12\}$.
\item[(2)]
If $n =2m \equiv 2 \Mod{4}$, and $m \ge 3$, then the class number $h_{\Delta}$ is even, regardless of whether $\OO_\Delta$ is maximal or non-maximal. 
For $m=3$ with $n =6$, we have $\Delta=32$ and $h_{\Delta}=1$.
\end{itemize} 
\end{thm}

\begin{proof}
(1) We first treat small $m$. 
First, $\Delta=-4$, coming from $(m,n)=(0,0)$, is a fundamental discriminant and has $h_\Delta=1$.
Next, $\Delta= 12$, coming from $(m,n)=(2,4)$, is a fundamental discriminant and has $h_\Delta=1$.
These give the two solutions listed in (1).
Next, $\Delta=0$, coming from $(m,n)=(1,2)$, is ruled out as not being a quadratic discriminant.
Finally, $\Delta=32$, coming from $(m,n)=(3,6)$, is a non-fundamental 
discriminant and has $h_\Delta=1$, giving the discriminant listed in (2).

In what follows it remains to consider  $m \ge 4$, so $\Delta \ge 60$.
By hypothesis $n=2m$, so $\Delta= 4m^2-4 = 4(m^2-1) $. 
We will show all such class numbers $h_\Delta$ are even, so $h_\Delta \ne 1$.
There are two cases.\\

{\bf Case 1.} {\it $m=2k$ is even, so $n =4k$, with $k \ge 2$.}\\

In this case we consider only fundamental discriminants.
Then $\Delta= 4 (4k^2-1) = \mbox{$4(2k+1)(2k-1)$}$. Let $g = \gcd(2k+1, 2k-1)$.
Then $g$ is odd and $g \divides 2$, so $g=1$. Now $2k+1 \ge 5$ has an odd prime factor,
call it $p$, and $2k-1 \ge 3$ has an odd prime factor $q$ distinct from $p$,
by relative primality. Thus $4 pq \divides \Delta$.

We show that $2 \divides h_{\Delta}$
when $\Delta$ is a fundamental discriminant, 
by excluding all cases of odd class number 
in  criterion (3) of \Cref{prop:HK}. Conditions $(a), (b)$ are ruled
out since $p q \divides \Delta$, condition $(c)$ is ruled out since $4 \divides \Delta$, and
condition  $(d)$ is ruled
out since $\Delta \ge 60$. Case 1 follows.\\

{\bf Case 2.} {\it $m=2k+1$ is odd, so $n=4k+2$, with $k \ge 2$.}\\

In this case we consider all real quadratic discriminants
$\Delta = 4(m^2-1)$. Since $m$ is odd, $x^2 \equiv 1 \Mod{8}$, so $32 \divides \Delta$,
and $\Delta$ is a quadratic discriminant.
Now the class number $h_\Delta$ is even (whether $\Delta$ is fundamental or non-fundamental)
by checking the criterion (2) of \Cref{prop:HK} for wide class number being odd.
Indeed $32 \divides \Delta$ rules out conditions (a)--(d), and case (e) is ruled out since $\Delta \ge 60$.
Case 2 follows.\\

(2) The Case 2 argument showed $h_\Delta$ is even for fundamental or non-fundamental $\Delta$.
\end{proof} 

\begin{proof}[Proof of \Cref{thm:M-classno-1}]
The result follows by combining the four cases \Cref{thm:Biro}, \Cref{thm:RQO1}, \Cref{thm:BKL}, and \Cref{thm:RQO2}.
\end{proof} 

For narrow Richaud--Degert maximal orders
of class number one, the remaining case 
not covered by the results above consists
of all squarefree discriminants $\Delta= 4n^2+1$, coming from case (2) of \Cref{prop:R--D}).
This case is Chowla's conjecture, which was formulated in 1976 (originally only for prime $\Delta$) in \cite[p.~48]{ChowlaF:1976}. It was settled by Bir\'{o} \cite[p. 179]{Biro03b} in 2003 using his method. 

\begin{thm}[Bir\'{o}]\label{thm:Biro2} 
If $\Delta = 4n^2+1$ is the 
discriminant of a maximal order having class number $h_\Delta=1$, then necessarily $n \le 1861$.
The solutions are $n \in \{1, 2, 3, 5, 7, 13 \}, $ with discriminants $\Delta \in \{5, 17, 37, 101, 197, 677\}.$
\end{thm}

The condition that $\Delta=4n^2+1$ is the discriminant of a
maximal order is equivalent to the requirement that $4n^2+1$ be squarefree,
used in Bir\'{o}'s formulation of the result.
The solution list in \Cref{thm:Biro2} was determined numerically using the bound $n \le 1861$.

\Cref{thm:Biro2} will be used in the non-maximal order case treated
in \Cref{subsec:42a}. 

\subsection{Unit-generated non-maximal orders with class number one}\label{subsec:42a}

We classify all non-maximal unit-generated quadratic orders $\OO$ of class number one.
For such orders $\OO$, all their over-orders, including the maximal order, necessarily have class number one. 
We must show an additional key property:\ If a non-maximal quadratic order $\OO$ has class number one and is unit-generated, then all its over-orders $\OO'$ necessarily are also unit-generated, with one exceptional case that is ``almost'' unit-generated.

We handle the case of prime relative conductor $\left[\OO' : \OO\right] = p$ first.

\begin{prop}\label{prop:ug-suborder-prime}
    Suppose that $\Delta = p^2\Delta'$, where $\Delta$ and $\Delta'$ are non-square discriminants and 
    $p$ is a prime number. If $\OO_\Delta$ is a unit-generated order and $h_\Delta = h_{\Delta'}$, then one of the following is true:
    \begin{itemize}
        \item[(1)] $\OO_{\Delta'}$ is a unit-generated order and $p < \frac{\log(p\sqrt{\Delta'}+1)}{\log(\sqrt{\Delta'}-1)} + 1$, from which it follows that $p \leq 17$. 
        \item[(2)] $p=2$ and $\Delta' \equiv 1 \Mod{8}$.
    \end{itemize}
\end{prop}
\begin{proof}
    Note that the quotient ideal $\colonideal{\OO_\Delta}{\OO_{\Delta'}} = p\OO_{\Delta'}$. We have by \cite[Thm.~6.5]{kopplagarias} the exact sequence
    \begin{equation}\label{eqn:exact}
        1 
        \to 
        \frac{\OO_{\Delta'}^\times}{\OO_\Delta^\times}
        \xrightarrow{\iota}
        \frac{\left(\OO_{\Delta'}/p\OO_{\Delta'}\right)^\times}{\left(\OO_{\Delta}/p\OO_{\Delta'}\right)^\times}
        \xrightarrow{\phi}
        \Cl(\OO_{\Delta})
        \xrightarrow{\psi}
        \Cl(\OO_{\Delta'})
        \to
        1.
    \end{equation}
    The equality $h_\Delta = h_{\Delta'}$ implies that $\psi$ is an isomorphism, so $\phi = 0$ and $\iota$ is an isomorphism.
    There is a ring isomorphism
    $\OO_{\Delta'}/p\OO_{\Delta'} \isom \Z[x]/(p,g(x)) \isom \F_p[x]/(g(x))$ where $g(x) = x^2 - \Delta' x + \frac{(\Delta')^2-\Delta'}{4}$ having $\disc(g) = \Delta'$. 
     Set $\omega = \frac{\Delta' + \sqrt{\Delta'}}{2}$. The subring $\OO_{\Delta}/p\OO_{\Delta'}$ 
     then has a compatible isomorphism $\OO_{\Delta}/p\OO_{\Delta'} = \frac{p\omega\Z + \Z}{p\omega\Z + p\Z} \isom \F_p$.
    The quotient of unit groups is determined by the Kronecker symbol $\left(\frac{\Delta'}{p}\right)$:
    \begin{equation}
        \frac{\left(\OO_{\Delta'}/p\OO_{\Delta'}\right)^\times}{\left(\OO_{\Delta}/p\OO_{\Delta'}\right)^\times}
        \isom \frac{\left(\F_p[x]/(g(x))\right)^\times}{\F_p^\times}
        \isom \begin{cases}
            \Z/(p-1)\Z & \text{if } \left(\frac{\Delta'}{p}\right) = 1, \\
            \Z/p\Z & \text{if } \left(\frac{\Delta'}{p}\right) = 0, \\
            \Z/(p+1)\Z & \text{if } \left(\frac{\Delta'}{p}\right) = -1.
        \end{cases}
    \end{equation}
    It follows from the fact that $\iota$ in \eqref{eqn:exact} is an isomorphism that the index
    \begin{equation}\label{eq:unitindex}
        [\OO_{\Delta'}^\times : \OO_\Delta^\times] = p - \left(\frac{\Delta'}{p}\right).
    \end{equation}
    If this index is not $1$, then $\OO_{\Delta'}$ contains a unit that is not in $\OO_\Delta$, so since $\OO_\Delta$ is unit-generated 
    and there are no other orders between $\OO_{\Delta'}$ and $\OO_{\Delta}$, we conclude that $\OO_{\Delta'}$ is unit-generated. If the index is $1$, then $p=2$ and $\left(\frac{\Delta'}{p}\right) = 1$, that is, $\Delta' \equiv \pm 1 \Mod{8}$, so $\Delta' \equiv 1 \Mod{8}$ because it is a discriminant. 

    It remains to show the upper bound on $p$ in case (1). Let $u > 1$ be a generator for $\OO_{\Delta'}^\times/\{\pm 1\}$ and 
    let $v > 1$ be a generator for $\OO_\Delta^\times/\{\pm 1\}$. 
    Let $\ep = \left(\frac{\Delta'}{p}\right)$. 
    By \eqref{eq:unitindex}, $v = u^{p-\ep}$. Moreover,
    \begin{align}
        \sqrt{\Delta'} - 1 = u - \sigma(u) - 1 &< u < u - \sigma(u) + 1 = \sqrt{\Delta'} + 1, \\
        p\sqrt{\Delta'} - 1 = v - \sigma(v) - 1 &< v < v - \sigma(v) + 1 = p\sqrt{\Delta'} + 1,
    \end{align}
    for the nontrivial Galois automorphism $\sigma$. Thus,
    \begin{align}\label{eq:pboundintermediate}
        p &= \frac{\log v}{\log u} + \ep < \frac{\log(p\sqrt{\Delta'}+1)}{\log(\sqrt{\Delta'}-1)} + 1. 
    \end{align} 
    Finally, note that the function $f(x,y) = \frac{\log(xy+1)}{\log(y-1)}$ is a decreasing function of $y$ for any $x,y > 1$, because
    \begin{align}
    	\frac{\partial}{\partial y} f(x,y) 
    	&= \frac{xy\log(y-1) - x\log(y-1) - xy\log(xy+1)- \log(xy+1)}{(y-1)(xy+1)\left(\log(y-1)\right)^2} \\
    	&< \frac{xy\log(xy+1) - x\log(y-1) - xy\log(xy+1)- \log(xy+1)}{(y-1)(xy+1)\left(\log(y-1)\right)^2} \\
    	&= \frac{- x\log(y-1) - \log(xy+1)}{(y-1)(xy+1)\left(\log(y-1)\right)^2} < 0.
    \end{align}
    From \eqref{eq:pboundintermediate}, we can conclude that
    \begin{align}\label{eq:pboundintermediate2}
    	p < f(p, \sqrt{\Delta'}) + 1 \leq f(p, \sqrt{5}) + 1.
    \end{align}
    By again taking a derivative, we can see that the function $g(x) = f(x, \sqrt{5}) + 1 - x$ is decreasing for $x>4.3$; also, $g(19) \approx -0.2 < 0$. Thus, $g(p) > 0$ (that is, \eqref{eq:pboundintermediate2}) implies $p \leq 17$.
\end{proof}

We now prove the key property for general relative conductor $\left[\OO' : \OO\right] = f$.

\begin{prop}\label{prop:ug-suborder}
    Suppose that $\Delta = f^2\Delta'$, where $\Delta$ and $\Delta'$ are non-square discriminants. 
    If $\OO_\Delta$ is a unit-generated order and $h_\Delta = h_{\Delta'}$, then one of the following is true:
    \begin{itemize}
        \item[(1)] $\OO_{\Delta'}$ is a unit-generated order.
        \item[(2)] $\OO_{\Delta'}$ is not a unit-generated order, 
        but $\OO_{4 \Delta'}$ is a unit-generated order, $2 \divides f$, and \mbox{$\Delta' \equiv 1 \Mod{8}$}.
    \end{itemize}
\end{prop}
\begin{proof}
    Let $e$ be the smallest divisor of $f$ such that $\OO_{e^2\Delta'}$ is unit-generated. If $e=1$, we are done;
    this is case (1). 
    
    Otherwise, suppose  $e \ge 2$ so $\OO_{\Delta'}$ is not a unit-generated order.
    Let $p$ be a prime divisor of $e$, and choose $p$ to be odd unless $e$ is a power of two. 
    We must have $h_{\Delta'} \leq h_{(e/p)^2\Delta'} \leq h_{e^2\Delta'} \leq h_\Delta$, 
    but also $h_{\Delta} = h_{\Delta'}$, so in fact $h_{\Delta'} = h_{(e/p)^2\Delta'} = h_{e^2\Delta'} = h_\Delta$. 
    By \Cref{prop:ug-suborder-prime} applied to the pair of discriminants $(e^2\Delta', (e/p)^2\Delta')$, we conclude that either:
    \begin{itemize}
        \item[(1')] $\OO_{(e/p)^2\Delta'}$ is a unit generated order, or
        \item[(2')] $p=2$ and $(e/p)^2\Delta' \equiv 1 \Mod{8}$.
    \end{itemize}
    The minimality of $e$ rules out case (1'). Thus, $p=2$, which forces $e$ to be a power of two, 
    and $(e/p)^2\Delta' \equiv 1 \Mod{8}$, which subsequently forces $e=2$ and $\Delta' \equiv 1 \Mod{8}$. 
    We conclude that $\OO_{4\Delta'}$ is a unit-generated order, $2 \divides f$, and $\Delta' \equiv 1 \Mod{8}$.
    This is case (2).
\end{proof}

\begin{proof}[Proof of  \Cref{thm:M-classno-2}]
    It may be checked directly 
    that each of these 
    (non-maximal)
    orders are unit-generated and have class number $1$. We 
    focus on the proof of the converse, that every unit-generated real quadratic 
    non-maximal
    order of class number $1$ is contained in the set $\{32,45,117\} \cup \{20,68,125\}$.

    Suppose that $\OO_\Delta$ is a unit-generated  order and $h_\Delta = 1$.
    Let $\Delta_0$ be the associated fundamental discriminant, and write $\Delta = f^2\Delta_0$, with $f >1$.
    We must have $h_{\Delta_0} = 1$. By \Cref{prop:ug-suborder}, one of the following holds.
    \begin{itemize}
        \item[(1)] $\OO_{\Delta_0}$ is a unit-generated order.
        \item[(2)] $\OO_{\Delta_0}$ is not a unit-generated order, but
         $\OO_{4\Delta_0}$ is a unit-generated order, $2 \divides f$, and $\Delta_0 \equiv 1 \Mod{8}$.
    \end{itemize}
   
    In case (1), \Cref{thm:M-classno-1} implies that the fundamental discriminant $\Delta_0$ is contained in the list
    \begin{equation}
        \Delta_0 \in \{5,8,12,13,21,29,53,77,173,293,437\}.
    \end{equation}
    For each fundamental $\Delta_0$ in this list, our results allow us to enumerate all possible non-fundamental values for $\Delta$ by a finite computation described as follows.
    \begin{enumerate}
        \item[(S1)] Check whether $\OO_{4\Delta_0}$ is a unit-generated order of class number $1$. (This happens for $\Delta_0 \in \{5, 8\}$,
        giving $\Delta \in \{20, 32\}$.) 
        \item[(S2)] For those cases where (S1) holds, also check whether $\OO_{4p^2\Delta_0}$ is a unit-generated order of class number $1$ for each of the finite list of primes satisfying 
        $p \leq 17$.
        (This never happens.)
        \item[(S3)] Thus, by \Cref{prop:ug-suborder-prime}, there do not exist any primes $p$ such that $\OO_{4p^2\Delta_0}$ is a unit-generated order of class number $1$.
        \item[(S4)] If $\Delta$ is even and $\Delta \neq 4\Delta_0$, applying 
        the criterion of \Cref{prop:ug-suborder} to a pair of the form $(\Delta, 4p^2\Delta_0)$ rules out $\Delta$.
        \item[(S5)] Check whether $\OO_{p^2\Delta_0}$ is a unit-generated order of class number $1$ for each of the finite list of odd primes satisfying 
        $p \leq 17$.
        (This happens for $(\Delta_0, p) \in \{(5,3), (13,3), (5,5)\}$, giving $\Delta\in \{45, 117, 125\}$.)
        \item[(S6)] For those cases where (S5) holds, also check whether $\OO_{p^2q^2\Delta_0}$ is a unit-generated order of class number $1$ for each of the finite list of odd primes satisfying 
		$q \leq 17$, allowing $q=p$.
        (This never happens.)
        \item[(S7)] Thus, by \Cref{prop:ug-suborder-prime}, there do not exist any primes $p,q$ such that $\OO_{p^2\Delta_0}$ or $\OO_{p^2q^2\Delta_0}$ is a unit-generated order of class number $1$, except for those found in (S5).
        \item[(S8)] If $\Delta$ is odd, $\Delta \neq \Delta_0$, and $\Delta$ was not found in (S5), applying \Cref{prop:ug-suborder} to a pair of the form $(\Delta, p^2\Delta_0)$ or $(\Delta, p^2q^2\Delta_0)$ to rule out $\Delta$. In the case \Cref{prop:ug-suborder}(2), we are using the conclusion of (S4) to say that $\OO_{4p^2\Delta_0}$ and $\OO_{4p^2q^2\Delta_0}$ cannot be unit-generated orders of class number $1$. 
    \end{enumerate}
    This computation yields for $\Delta= n^2-4$ with $n\in \{ 6, 7, 11\}$ the non-maximal discriminants $\Delta \in \{32,45,117\}$.
    It yields for  $\Delta= n^2+4$  with  $n \in \{4, 11\}$ the non-maximal discriminants $\Delta \in  \{ 20, 125\}$.

    In case (2), $\Delta_0 \equiv 1 \Mod{8}$,
    $\OO_{\Delta_0}$ is not a unit-generated order, and $4\Delta_0 = n^2 \pm 4$ for some integer $n$. Write $n = 2m$. 
    If $4\Delta_0 = n^2 - 4$, then $\Delta_0 = m^2 - 1 \equiv 0,3,7 \Mod{8}$, a contradiction.
    Thus $4\Delta_0 = n^2 + 4$, so $\Delta_0 = m^2 + 1$. 
    Necessarily $m$ must be even, say $m=2k$, so $\Delta_0 = 4k^2 + 1$. 
    Since $\Delta_0$ is fundamental and $h(\Delta_0)=1$, \Cref{thm:Biro2} 
    applies to give $k \in \{1,2,3,5,7,13\}$, so
    \begin{equation}
        \Delta_0 \in \{5,17,37,101,197,677\}.
    \end{equation}
    Only $\Delta_0 =17$ satisfies $\Delta_0 \equiv 1 \Mod{8}$. 
    It yields for $n=8$ the new non-maximal discriminant $\Delta = 68$.
    There are no more, as one  rules out all $\Delta = 4p^2 \Delta_0$ for $p \leq 17$
    being unit-generated orders of class number one.
\end{proof}

\begin{rmk}\label{rem:powers-of-5}
The computation  must test the unit-generated order property.
For $\Delta_k= 5^{2k+1},$ we have $h_{\Delta_k}=1$ for all $k \ge 0$.
However only $\Delta_0= 5$ and $\Delta_1=125$ are unit-generated. 
Here $\Delta_2=3125$ in Step (S5) fails to be unit generated. 
\end{rmk}

We summarize the results for unit-generated quadratic orders of class number one 
in \Cref{table:1}.

\bgroup
\def\arraystretch{1.1}
\begin{table}[h!]
\begin{tabular}{| c | c | c | }
\hline
$n$ & $\Delta= n^2- 4$ & $f_{\Delta}$\\
\hhline{|=|=|=|}
$0$ & $-4$ & $1$  \\
$1$ & $-3$ & $1$ \\
$3$ & $5$ & $1$ \\
$4$ & $12$ & $1$   \\
$5$ & $21$ & $1$ \\ 
$6$ & $32$ & $2$  \\
$7$ & $45$ & $3$ \\
$9$ & $77$ & $1$  \\
$11$ & $117$ & $3$   \\
$21$ & $437$ & $1$ \\
\hline
\end{tabular}
\quad \quad \quad
\begin{tabular}{| c | c | c | }
\hline
$n$ & $\Delta= n^2+4$ & $f_{\Delta}$  \\
\hhline{|=|=|=|}
$1$ & $5$ & $1$ \\
$2$ & $8$ & $1$  \\
$3$ & $13$ & $1$ \\
$4$ & $20$ & $2$  \\
$5$ & $29$ & $1$ \\
$7$ & $53$ & $1$  \\
$8$ & $68$ & $2$  \\
$11$ & $125$ & $5$  \\
$13$ & $173 $ & $1$ \\
$17$ & $293$ & $1$ \\
\hline
\end{tabular}
\medskip
\caption{Quadratic order discriminants $\Delta=n^2+r = (f_{\Delta})^2\Delta_0$ with $r \in \{\pm 4\}$ and class number $h_\Delta=1$.
}
\label{table:1}
\end{table}
\egroup
There are $11$  real quadratic  fields having unit-generated maximal orders having class number one, compared with $9$ imaginary quadratic
fields having  maximal order having  class number one.
There are $17$ unit-generated real quadratic orders having class number one, compared with $15$ imaginary quadratic
orders having class number one.

\section{Unit-generated orders with one class per genus} \label{sec:5}

We treat ideal class groups for unit-generated orders that have one class per genus.
Technically we treat a more general problem, that of determining all unit-generated orders whose 
(wide) class groups
are of exponent $2$, that is, $\Cl(\OO_{\Delta}) = \Cl(\OO_{\Delta})[2]$. 

\subsection{Gauss one class per genus problem}\label{subsec:51}

A famous problem of Gauss is the one class  per genus problem for
positive definite (primitive) integral binary quadratic forms 
of negative discriminant $\Delta = -4D$. 
\begin{enumerate}
\item[(1)]
In {\em Disquisitiones Arithmeticae}, Gauss \cite[Article 304]{Gauss1801} gave an empirical list of $65$
discriminants\footnote{Gauss \cite{Gauss1801} used the determinant $D=ac-b^2 = - \frac{1}{4}\Delta$,
rather than the discriminant.} of positive definite binary quadratic forms (restricted to  $\Delta \equiv 0 \Mod{4}$) having one class per genus,
and asked if the list was complete.  

\item[(2)]
Gauss noted that his list  agreed with Euler's list of $65$ ``idoneal numbers'' $D$ (taking $\Delta = -4D$),
which he knew of from a published letter of Euler  \cite{Euler1778}. 
Euler's concept of ``numerus idoeneus'' (convenient numbers)
was: An integer $D \geq 1$ is \textit{idoneal} if, whenever $m$ is an odd integer and $x^2+Dy^2=m$ 
has a unique solution in coprime integers $x,y \geq 0$, it follows that $m$ is prime; see \cite{Euler1778b,Kani11}. 
Euler found $65$ positive integers $D$ with this property.
The equivalence of the Euler definition  to the one class per genus problem  of Gauss for $\Delta \equiv 0 \Mod{4}$ was proved by Grube \cite{Grube1874}
in 1874, as described in Kani \cite{Kani11}; see also Frei \cite{Frei1984}.

\item[(3)] Work of Leonard Eugene Dickson extended the list of {discriminants} to include odd discriminants $\Delta \equiv 1 \Mod{4}$, 
finding $36$ discriminants $\Delta \equiv 1 \Mod{4}$ having one class per genus.
There are $101$ discriminants in the two lists combined; see \cite{BriggsC:54}.
\end{enumerate}

The Gauss one class per genus problem for binary forms of negative discriminant remains unsolved. 
See Stark \cite{Stark:07} for information on its status and Lemmermeyer \cite{Lemmermeyer:07} for further history of genus theory.

The one class per genus problem for binary quadratic forms 
was carried over  to ideals of orders of quadratic number fields, 
by Dedekind \cite{Dedekind:1877,Dedekind:1877a,Dedekind:1894}.
For  the genus theory of real quadratic orders, one must use the narrow class group $\Cl^{+}(\OO_{\Delta})$
rather than the class group $\Cl(\OO_{\Delta})$;
see \Cref{subsec:52}.

One may view the one class per genus question for  unit generated orders of real quadratic fields
as an analogue of the Gauss one class per genus  problem, either with or without
the restriction to $\Delta \equiv 0 \Mod{4}$.

\subsection{Finiteness of unit-generated orders with one class per genus}\label{subsec:52} 

We recall facts about 
the genus theory of Gauss for general binary quadratic forms, described in \cite[Chapter 6.5]{Halter-Koch13},
and its relation to  the narrow  class group $\Cl^{+}(\OO_{\Delta})$ and
(wide) class group $\Cl(\OO_{\Delta})$ 
of a quadratic order $\OO_{\Delta}$
of the same discriminant.
\begin{enumerate}
\item
The integral binary quadratic form class group $\mathfrak{F}_{\Delta}$ of
$\SL_2(\Z)$-classes of primitive binary quadratic forms of discriminant $\Delta$, having group law given by Gauss composition of forms, is isomorphic to 
the narrow ideal class group $\Cl^{+}(\OO_{\Delta})$; see \cite[Theorem 6.4.2]{Halter-Koch13}.
\item
There is a surjective homomorphism $ \Cl^{+}(\OO_{\Delta}) \to \Cl(\OO_{\Delta})$,
with kernel of size $1$ or $2$.
\item 
The group of (Gauss) genera of quadratic forms is $\mathfrak{F}_{\Delta}/ \mathfrak{F}_{\Delta}^2$ and is isomorphic to $\Cl^{+}(\OO)/ \Cl^{+}(\OO)^2$, which we call the \textit{genus group}. 
The group $\Cl(\OO)/\Cl(\OO)^2$ is a quotient of the genus group, with quotient map from
$\Cl^{+}(\OO)/ \Cl^{+}(\OO)^2$ to $\Cl(\OO)/\Cl(\OO)^2$ having
kernel of size $1$ or $2$.
(The statement that the principal genus is $\Cl^{+}(\OO)^2$ is sometimes
called the {\em principal genus theorem}; see \cite{Lemmermeyer:07}.)
\end{enumerate}

The group of genera of a real quadratic order has order $2^{r+j}$, where $r$ is the number of distinct odd primes dividing the discriminant $\Delta$,
and $-1 \le j \le 1$. To state its order precisely,
write $\Delta= \sgn(\Delta)2^{e_0} p_1^{e_1} \cdot \ldots \cdot p_r^{e_r}$, with 
$e _0 \ge 0, e_i \ge 1$, for $i \ge 1$,
and define $\mu(\Delta)$ by
\begin{equation}
\mu(\Delta) = 
\begin{cases}
r & \quad \mbox{if} \quad \Delta \equiv 1 \Mod{4} \,\,\, \mbox{or} \,\,\, \Delta \equiv 4 \Mod{16}, \\
r+1 & \quad \mbox{if} \quad \Delta \equiv 8 \,\, \mbox{or} \,\, 12 \Mod{16} \,\,\, \mbox{or} \,\,\, \Delta \equiv 16 \Mod{32}, \\
r+2 & \quad \mbox{if} \quad \Delta \equiv 0 \Mod{32}.
\end{cases} 
\end{equation} 
Then its order is given by \cite[Theorem 6.5.2]{Halter-Koch13}, which we state as \Cref{prop:genusorder}.

\begin{prop}\label{prop:genusorder}
If $\Delta$ is a (real or imaginary) quadratic discriminant, 
then 
\begin{equation}
\abs{\mathfrak{F}_\Delta[2]} = \abs{\mathfrak{F}_{\Delta}/ \mathfrak{F}_{\Delta}^2} = 2^{\mu(\Delta)-1}.
\end{equation}
\end{prop} 

We upper-bound $\mu(\Delta)$ in terms of $\abs{\Delta}$.
We have $\mu(\Delta)-1 \le \omega(\Delta)$,
where $\omega(m)$ counts the number of distinct prime divisors of $m$ (without multiplicity).
It is well known that $\omega(m) = O\!\left(\frac{\log m}{\log\log m}\right)$ for $m \ge 2$ 
\cite[Theorem 2.10]{MV},
so we conclude, for $\abs{\Delta} \ge 20$, 
\begin{equation}\label{eqn:omega-bound}
\mu(\Delta)-1 = O\!\left(\frac{\log \abs{\Delta}}{\log\log \abs{\Delta}}\right).
\end{equation}
The constant in the $O$-notation is effectively computable. 
 
\begin{proof}[Proof of \Cref{thm:twotorsion}]
We first consider arbitrary quadratic discriminants $\Delta$ (not necessarily unit-generated) such that $\Cl(\OO_{\Delta})$ is $2$-torsion.
Then by \Cref{prop:genusorder} we have
\begin{equation}\label{eqn:one-class-genus-bound0}
\log \abs{\Cl(\OO_{\Delta})} \le \log \abs{\Cl^{+}(\OO_{\Delta})} = \log\abs{\mathfrak{F}_\Delta[2]} 
= \log 2^{\mu(\Delta)-1}
\end{equation}
Using  \eqref{eqn:omega-bound}, we obtain, for $\abs{\Delta} \ge 20$, 
\begin{equation}\label{eqn:one-class-genus-bound}
\log \abs{\Cl(\OO_{\Delta})} = O\!\left(\frac{\log \abs{\Delta}}{\log\log \abs{\Delta}}\right).
\end{equation}
From \eqref{eqn:one-class-genus-bound} it follows that, for any infinite family $\mathcal{F}$ of 
real quadratic  orders in which $\Cl(\OO_\Delta)$ is $2$-torsion, we necessarily have 
\begin{equation}
\log \abs{\Cl(\OO_{\Delta})} = o(\log \Delta) \quad \mbox{as} \quad \Delta \to \infty, \quad \mbox{with} \quad \Delta \in \mathcal{F}.
\end{equation}
For real quadratic orders such infinite families do exist since, as noted earlier, there exist infinitely many real quadratic orders having class number $1$.
 
Now we restrict to unit-generated orders. To show there are only finitely many unit-generated orders $\OO_\Delta$ 
such that $\Cl(\OO_\Delta) = \Cl(\OO_\Delta)[2]$,
we argue by contradiction. 
So suppose there are infinitely many such $\Delta_n^\pm > 0$ in a family $\mathcal{F}= \{n: n = n_i, \,i \ge 1\}$, with $n_1 < n_2 < n_3 < \cdots$. 
For unit-generated orders we have $\Delta_n \le n^2 + 4$, so $\log \Delta_n= O(\log n)$, and by the argument above, 
\begin{equation}
\log \abs{\Cl(\OO_{\Delta_{n}})} = o(\log n) \quad \mbox{as} \quad k \to \infty, \quad \mbox{with} \quad n=n_i \in \mathcal{F}.
\end{equation}
This limiting behavior contradicts 
Theorem \ref{thm:bs00}(1), which asserts for unit-generated orders that 
\begin{equation}\label{eqn:ineffective2}
\log \abs{\Cl(\OO_{\Delta_n})} = \log n + o(\log n) \quad \mbox{as} \quad n \to \infty.
\end{equation}
We conclude there must be  only finitely many unit-generated real quadratic orders 
that have a (wide) class group that is  $2$-torsion.
\end{proof}

\subsection{Computations of unit-generated orders having (wide or narrow) class group of exponent $2$}\label{subsec:53}

We give lists of all unit-generated quadratic orders having wide class group of exponent $2$ 
(that is, $\Cl(\OO_{\Delta})= \Cl(\OO_{\Delta})[2]$) and having discriminant $\Delta< 10^{10}$.
We also list the subset of those unit-generated orders having narrow class number of exponent $2$ 
(that is, $\Cl^+(\OO_{\Delta})= \Cl^+(\OO_{\Delta})[2]$, or ``one class per genus'').  
We find a total of $86$ unit-generated quadratic orders having $2$-torsion (wide) class group and discriminant $\Delta < 10^{10}$, 
which we list in \Cref{table:2A} and \Cref{table:2B}.
(The discriminant $\Delta=5$ again appears twice.) 
In the case $\Delta=n^2+4$ in  \Cref{table:2B}, all wide class groups coincide with narrow class groups, because the fundamental unit has norm $-1$.

These tables include all known unit-generated quadratic orders having one class per genus. In the case $\Delta=n^2+4$,
all discriminants in \Cref{table:2B} have one class per genus, since
\mbox{$\Cl^{+}(\Delta) \isom \Cl(\Delta)$} because their generating unit has norm $-1$.
In \Cref{table:2A}, only the rows with \mbox{$\Cl^{+}(\OO_{\Delta}) \isom (\Z/2\Z)^j$}
have one class per genus.

We computed the data for \Cref{table:2A} and \Cref{table:2B} using \texttt{SageMath}. Our code is publicly available on \texttt{GitHub} at \url{https://github.com/gskopp/UnitGeneratedOrders}.

\bgroup
\begin{small}
\def\arraystretch{1.1}
\begin{table}[h!]
\begin{tabular}{| c | c | c | c | }
\hline
$\Cl(\OO_\Delta)$ & $\Cl^+(\OO_\Delta)$ & $\Delta= n^2 -4$ & $f_{\Delta}$ \\
\hhline{|=|=|=|=|}
$1$ & $1$ & $-4, -3, 5$ & $1$  \\
\hline
$1$ & $\Z/2\Z$ & $12, 21, 77, 437$ & $1$  \\
    & & $32$ & $2$ \\
    & & $45, 117$ & $3$  \\
\hline
$\Z/2\Z$ & $\left(\Z/2\Z\right)^2$ & $60, 140, 165, 285, 357, 572, 957, 1085, 2397$ & $1$  \\
         & & $96$ & $2$ \\
         & & $252$ & $3$  \\
         & & $192$ & $4$  \\
         & & $525$ & $5$  \\
         & & $320$ & $8$  \\
\hline
$\Z/2\Z$ & $\Z/4\Z$ & $221, 1517$ & $1$  \\
                  & & $725$ & $5$ \\
\hline
$\left(\Z/2\Z\right)^2$ & $\left(\Z/2\Z\right)^3$ & $780, 1020, 1365, 1932, 2805, 4485, 5180, 7917, 8645$ & $1$  \\
                        & & $480, 672, 1760, 2912$ & $2$ \\
                        & & $1440$ & $6$  \\
                        & & $2112$ & $8$  \\
\hline
$\left(\Z/2\Z\right)^2$ & $\Z/2\Z \times \Z/4\Z$ & $3965, 7565$ & $1$  \\
\hline
$\left(\Z/2\Z\right)^3$ & $\left(\Z/2\Z\right)^4$ & $4620, 12540, 26565$ & $1$  \\
                        & & $3360, 7392, 14880, 19040, 23712, 27552$ & $2$  \\
                        & & $6720$ & $8$  \\
\hline
$\left(\Z/2\Z\right)^4$ & $\left(\Z/2\Z\right)^5$ & $68640$ & $2$ \\
\hline
\end{tabular}
\medskip
\caption{Discriminants $\Delta=n^2-4 = (f_{\Delta})^2\Delta_0$ 
having (wide) class group $\Cl(\OO_\Delta)=\Cl(\OO_\Delta)[2]$, 
complete for $\Delta < 10^{10}$. The list is not known to be complete without the restriction $\Delta < 10^{10}$.}
\label{table:2A}
\end{table}
\end{small}
\egroup
%

\bgroup
\begin{small}
\def\arraystretch{1.1}
\begin{table}[h!]
\begin{tabular}{| c | c | c | c |}
\hline
$\Cl(\OO_\Delta)=\Cl^+(\OO_\Delta)$ & $\Delta= n^2 + 4$ 
& $f_{\Delta}$  \\
\hhline{|=|=|=|}
$1$ & $5, 8, 13, 29, 53, 173, 293$ & $1$  \\
    & $20, 68$ & $2$ \\
    & $125$ & $5$  \\
\hline
$\Z/2\Z$ & $40, 85, 104, 365, 488, 533, 629, 965, 1448, 1685, 1853, 2813$ & $1$  \\
         & $260$ & $2$ \\
         & $200$ & $5$  \\
         & $845$ & $13$  \\
\hline
$\left(\Z/2\Z\right)^2$ & $680, 1160, 2120, 2405, 3485, 3848, 5480, 10205, 16133$ & $1$ \\
\hline
$\left(\Z/2\Z\right)^3$ & $8840, 21320, 32045$ & $1$  \\
\hline
\end{tabular}
\medskip
\caption{Discriminants $\Delta=n^2+4= (f_{\Delta})^2\Delta_0$ having wide  class group $\Cl(\OO_\Delta)=\Cl(\OO_\Delta)[2]$, complete for $\Delta < 10^{10}$. The list is not known to be complete without the restriction $\Delta < 10^{10}$.}
\label{table:2B}
\end{table}
\end{small}
\egroup

A list of discriminants having exponent $2$ class groups for maximal orders of extended Richaud--Degert type was previously computed by Loboutin, Mollin, and Williams \cite{LouboutinMW93}, complete under the assumption of the generalized Riemann hypothesis. We have checked that our tables match theirs where they
overlap.\footnote{
They list such orders in their Tables 1 and 3, which exclude those with class number $1$.
Their Tables 1 and 3 include maximal orders for $\Delta = n^2+4$, and their list matches our \Cref{table:2B} in the rows labeled $f=1$.
Their Table $1$ includes maximal orders with $\Delta=n^2-4$ and $2 \divides \Delta$ (with $D= \Delta/4$),
and it matches our \Cref{table:2A} for $f=1$ in these cases. 
All the odd discriminant cases, with $\Delta=D$, appear in their Table 3.}

\begin{rmk}\label{rem:im-quadratic}
For comparison with the imaginary quadratic case, the results of Frei \cite[Theor\'{e}me 4.5]{Frei1984} show
the largest (absolute) discriminant of the $65$ even discriminants is $-\Delta=2^2 \cdot 1848= 7392$, and the largest of the $36$
odd discriminants is $-\Delta=3315$.
\end{rmk}

\section{Concluding remarks}\label{sec:6aa}

The families of real quadratic orders with $\Delta = \Delta^+_n = n^2-4$ and $\Delta = \Delta^-_n = n^2+4$ behave like
imaginary quadratic fields in terms of the rate of growth of their class number. One may ask whether other class group statistics for these families
has similar behavior 
to the imaginary quadratic case. For example,
one such statistic is the average size of the $m$-torsion, for fixed $m$,
$\frac{1}{X}\sum_{n \leq X} \abs{\Cl(\OO_{\Delta^\pm_n})[m]}$ as $X \to \infty$.
The case $m=2$ reduces by genus theory to a statistic on the number of distinct prime divisors of $\Delta^\pm_n$, which seems approachable. 

\Cref{prob:2}, the classification of unit-generated quadratic orders whose class group is $2$-torsion,
subsumes a real quadratic analogue of the Gauss one class per genus problem, which itself is currently unsolved. 
The complete classification of unit-generated orders with $2$-torsion class group seems out of reach at this time, in parallel with the idoneal number problem being unsolved.

The definition of unit-generated orders makes sense for arbitrary number fields. 
For real quadratic, complex cubic, and totally complex quartic fields, which are the cases where the rank of the unit group is one, this definition implies ``small regulator'' and hence ``large class number.''
One may hope for analogues of results of this paper in the complex cubic and totally complex quartic cases.

\section*{Acknowledgements}\label{sec:ack}

Work of the first author was supported by NSF grant DMS-2302514. 
Work of the second author was partially supported by NSF grant DMS-1701576.

We thank Henri Darmon for suggesting the key idea behind the proof of \Cref{prop:ug-suborder-prime}.
We are grateful to the developers of \texttt{SageMath}, which we used to compute class numbers and other properties of real quadratic orders.

\bibliographystyle{abbrv}
\bibliography{references}

\end{document}